\documentclass[12pt]{article}
\RequirePackage{amsthm,amsmath,amsfonts,amssymb}
\usepackage{graphicx}
\usepackage{enumerate}
\usepackage{natbib}
\usepackage{url} 

\usepackage{natbib}
\usepackage{subcaption}

\newcommand{\blind}{1}

\usepackage{color}

\newcommand{\X}{\mathrm{X}}
\newcommand{\Y}{\mathrm{Y}}
\newcommand{\U}{\mathrm{U}}

\newcommand{\M}{\mathrm{M}}

\newcommand{\pr}{P}
\newcommand{\KM}{\mathrm{KM}}
\newcommand{\KPOD}{\mathrm{KPOD}}
\newcommand{\what}{\widehat}

\newcommand{\Rb}{\mathbb{R}}

\newtheorem{theorem}{Theorem}[section]
\newtheorem{prop}[theorem]{Proposition}
\newtheorem{lemma}[theorem]{Lemma}

\begin{document}

\def\spacingset#1{\renewcommand{\baselinestretch}%
{#1}\small\normalsize} \spacingset{1}


\if1\blind
{
  \title{\bf Some notes on \\the $k$-means clustering for missing data}
  \author{Yoshikazu Terada\thanks{
    Jointly affiliated at RIKEN Center for Advanced Integrated Intelligence Research}\hspace{.2cm}\\
    Graduate School of Engineering Science, Osaka University\\
    and \\
    Xin Guan \\
    Graduate School of Engineering Science, Osaka University}
  \maketitle
} \fi

\if0\blind
{
  \bigskip
  \bigskip
  \bigskip
  \begin{center}
    {\LARGE\bf Some notes on the $k$-means clustering\\ for missing data}
\end{center}
  \medskip
} \fi

\bigskip
\begin{abstract}
The classical $k$-means clustering requires a complete data matrix without missing entries. 
As a natural extension of the $k$-means clustering for missing data, 
the $k$-POD clustering has been proposed, which ignores the missing entries in the $k$-means clustering.
This paper shows the inconsistency of the $k$-POD clustering even under the missing completely at random mechanism.
More specifically, the expected loss of the $k$-POD clustering can be represented as 
the weighted sum of the expected $k$-means losses with parts of variables.
Thus, the $k$-POD clustering converges to the different clustering from the $k$-means clustering as the sample size goes to infinity.
This result indicates that although the $k$-means clustering works well, the $k$-POD clustering may fail to capture the hidden cluster structure.
On the other hand, for high-dimensional data, 
the $k$-POD clustering could be a suitable choice 
when the missing rate in each variable is low.
\end{abstract}

\noindent%
{\it Keywords:}  Missing data, $k$-means clustering, consistency. 
\vfill

\newpage
\spacingset{1.9} 

\section{Introduction}

The $k$-means clustering is one of the most famous clustering algorithms, which provides a partition minimizing the sum of within-cluster variances. 
When a given data matrix has missing entries, 
most clustering algorithms, including $k$-means clustering, 
cannot be applied, whereas missing data are common in various real data applications, e.g., \citep{wagstaff2005making}. 
Here, we consider the $k$-means clustering for missing data.
There are two common approaches for handling missing entries: \textit{complete-case analysis} and \textit{imputation}, both of which can be used as pre-processing ways before the clustering \citep{himmelspach2010clustering}. 
However, since we delete all data points with missing entries in complete-case analysis, 
the number of complete cases could be too small in multivariate data. 
The imputation approach works well when the assumptions on a hidden probabilistic model are correct, 
while it is often complicated and its computational cost is high \citep{lee2022incomplete}. 
Another approach is to modify the Euclidean distance used in $k$-means clustering. 
For example, the partial distance that involves only the observed dimensions is a popular choice \citep{wagstaff2004clustering,lithio2018efficient,datta2018clustering}, 
the main problem is that the modified measurements for distance may not reflect the true structure based on all dimensions and may not even be a distance measure. 

As a natural extension of the $k$-means clustering for missing data, 
the $k$-POD clustering is proposed by \citet{chi2016k}, 
which can be considered as a special case of the matrix completion issue (e.g., see \cite{jain2013low}).
For a data matrix $\mathrm{{X}}=(X_{ij})_{n\times p}$, the set of indexes $\Omega\subset \{1,\dots,n\}\times \{1,\dots,p\}$ indicates the observed entries. The projection $\mathcal{P}$ onto an index set $\Omega$ is introduced to replace the missing entries with zero. That is, $[\mathcal{P}_{\Omega}(\mathrm{{X}})]_{ij}=X_{ij}$ if $(i,j)\in\Omega$, 0 otherwise. Further write a binary matrix $\mathrm{{U}}=(u_{il})_{n\times k}$ for the cluster membership, where $u_{il}=1$ if $i$th observation belongs to $l$th cluster. The $k$ cluster centres are denoted by a matrix $\mathrm{{M}}=(\mu_{lj})_{k\times p}$, the $l$th row of which represents the $l$th cluster center. 
Then, the loss function of the $k$-POD clustering is defined as
\[
\min_{\mathrm{U},\mathrm{M}} \| \mathcal{P}_{\Omega}(\mathrm{{X}}-\mathrm{{U}}\mathrm{M}) \|_F^2 \;\text{ such that }\; 
\mathrm{{U}}\in \{0,1\}^{n\times k},\;\text{ and }\; 
\sum_{l=1}^ku_{il} = 1 \;\text{ for all }\; i = 1,\dots,n,
\]
where 
$\|\mathrm{{A}}\|_F=\big(\sum_{i=1}^n\sum_{j=1}^pa_{ij}^2\big)^{1/2}$ denotes the Frobenius norm 
of a matrix $\mathrm{A} = (a_{ij})_{n\times p}$.
When $\Omega=\{1,\dots,n\}\times\{1,\dots,p\}$, the above loss is equivalent to that of the $k$-means clustering. 
Therefore, the $k$-POD clustering ignores the missing entries in the $k$-means clustering.
The simple and fast majorization-minimization algorithm can solve the optimization of the above loss.
The $k$-POD clustering stably performs 
even under a large proportion of missingness. 
The numerical experiments in \citet{chi2016k} show that 
the $k$-POD clustering works well under various cases. 
As mentioned in \citet{chi2016k}, 
we should note that the $k$-POD clustering has the common limitations as the $k$-means clustering. 
It still seems to work well in those settings where the $k$-means clustering works. 

Interestingly, \cite{WangEtAl19} independently proposes the following $k$-means clustering for missing data:
\begin{align*}
\min_{\Y,\U,\M}\|\mathrm{Y} - \U\M\|_F^2
\;\text{ such that }\; 
&\Y \in \mathbb{R}^{n\times p}: \mathcal{P}_{\Omega}(\Y) = \mathcal{P}_{\Omega}(\X),\\
&\mathrm{{U}}\in \{0,1\}^{n\times k},\;\text{ and }
\sum_{l=1}^ku_{il} = 1 \;\text{ for all }\; i = 1,\dots,n.
\end{align*}
This method is identical to the $k$-POD clustering since
the optimal solution of this problem should satisfy $\mathcal{P}_{\Omega^c}(\Y) = \mathcal{P}_{\Omega^c}(\U\M)$ where $\Omega^c$ is the complement of $\Omega$.

In this paper, unfortunately, 
we will show that the $k$-POD clustering provides
an essentially different partition of the data space with the $k$-means clustering
even under the simplest missing mechanism called the missing completely at random.
More precisely, 
the estimated partition by the $k$-POD clustering converges to 
a partition different from the limit of the $k$-means clustering (\citealp{pollard1981strong}) in the large sample limit.
Thus, even for the setting where the $k$-means clustering works, 
the $k$-POD clustering may fail to capture the hidden cluster structure.
Here, we note that the $k$-means clustering with complete cases has 
the same limit as the $k$-means clustering with all original data under the missing completely at the random (MCAR) mechanism.
To explain this problem, we demonstrate an illustrative example in Figure~\ref{fig:intro}. 
Grey data points are generated from a two-dimensional Gaussian mixture distribution ($n=10^4$), 
and the $i$th row $(X_{i1},X_{i2})$ is randomly observed with probabilities 
$(q_1,q_2) = (1/3,2/3)$, where $q_j\;(j=1,2)$ is the probability of the $j$th dimension being observed.
Here, the number of complete cases is approximately $2200$.
The green dotted line is the cluster boundary of the $k$-mean clustering with complete cases and is almost the same as the cluster boundary of $k$-means with all grey data points.
However, the cluster boundary of the $k$-POD clustering (blue dashed line) is 
completely different from these boundaries.
The essential reason for the distortion result of the $k$-POD clustering lies in the 
difference between the expected losses of the $k$-means and $k$-POD clusterings.
As shown later, 
the expected loss function of the $k$-POD clustering can be written as
the weighted sum of the expected losses of the $k$-means clustering using parts of variables.
 
\begin{figure}
     \centering
     \includegraphics[width=0.5\textwidth]{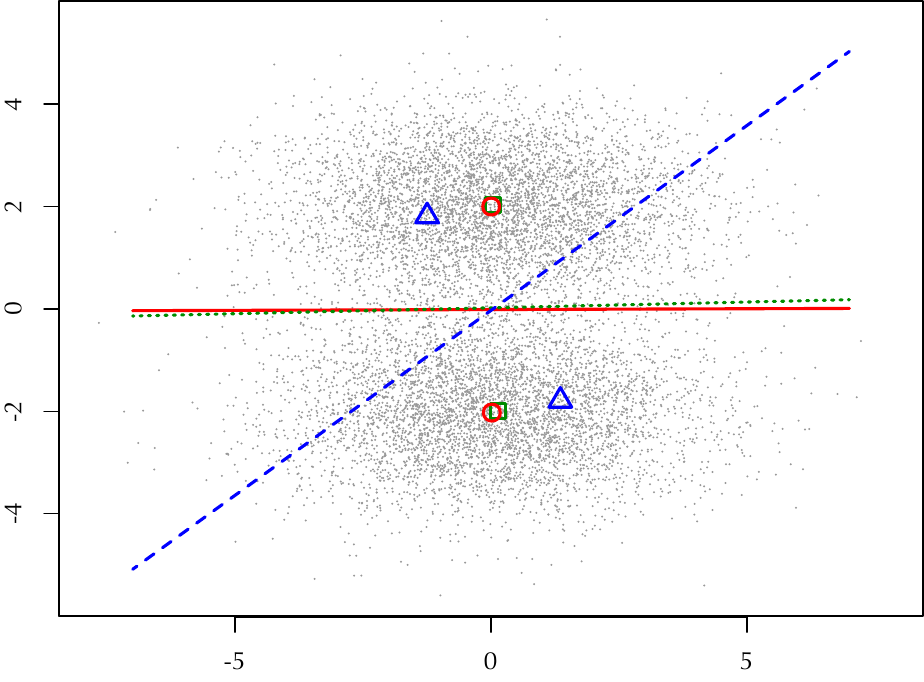}
     \caption{An illustrative example showing that the $k$-POD clustering fails.
     (solid and circle: $k$-means clustering with all data points,
      dotted and square: $k$-means clustering with complete cases,
      triangle and dashed: $k$-POD clustering with missing data).
     }
     \label{fig:intro}
\end{figure}

\section{Inconsistency of $k$-POD clustering}

Let $X_1,\dots,X_n$ be a $p$-dimensional independent sample from a population distribution $P$. 
Write $\X = (X_{ij})_{n\times p}$.
Here, we simply consider the missing completely at random mechanism.
Let $R_{ij}=1$ if $X_{ij}$ is observed and $R_{ij} = 0$ if $X_{ij}$ is missing. 
The response indicator vectors $R_i = (R_{i1},\dots,R_{ip})^{T}\;(i=1,\dots,n)$ are independently distributed from
the multinomial distribution and are completely independent with $\X$.
We assume that $\pr(R_{1}=1_p) >0$ where $1_p = (1,\dots,1)^{T}\in \mathbb{R}^p$.

Let $k$ be the number of clusters.
For given cluster centres $\M= (\mu_1,\dots,\mu_k)^{T}\in \Rb^{k\times p}$, 
the empirical loss of the $k$-means clustering is 
\[
\widehat{L}_n^{(\KM)}(\M) = \frac{1}{n}\sum_{i=1}^n\min_{1 \le l \le k} \|X_i - \mu_l\|^2
=
\frac{1}{n}\|\X - \U\M\|_F^2,
\]
where $\mu_l$ is the $l$th row of $\M$.
Using a similar calculation, the empirical loss of the $k$-POD clustering can be written as
\[
\widehat{L}_n^{(\KPOD)}(\M)
=
\frac{1}{n}\sum_{i=1}^n\min_{1 \le l \le k} \sum_{j=1}^pR_{ij}(X_{ij}-\mu_{lj})^2.
\]

Let $\what{\M}_\KM= \arg\min_\M \widehat{L}_n^{(\KM)}(\M)$ be the estimator of the $k$-means clustering.
\citet{pollard1981strong} shows that, as the sample size $n$ goes to infinity,
the estimator $\what{\M}_\KM$ converges to the minimizer of the expected loss of the $k$-means clustering, that is, 
\[
L^{(\KM)}(\M) = E\left[\min_{1 \le l \le k} \|X_1 - \mu_l\|^2 \right].
\]
Similarly, we can define the expected loss of the $k$-POD clustering as 
\[
L^{(\KPOD)}(\M) = E\left[\min_{1 \le l \le k}\sum_{j=1}^pR_{1j}(X_{1j}-\mu_{lj})^2\right].
\]

As the first result, 
we show that the expected loss of the $k$-POD clustering can be represented as 
the weighted sum of all possible expected $k$-means losses with parts of variables.

\begin{prop}\label{prop}
Let $R_{1}=(R_{11},\dots,R_{1p})^{T}$.
Under the above assumptions, 
\begin{align*}
L^{(\KPOD)}(\M) 
&=
\sum_{r \in \{0,1\}^p}\pr(R_{1} = r) L^{(\KM)}(\M\mid r),
\end{align*}
where $r = (r_1,\dots,r_p)^T\in \{0,1\}^p$, and 
$L^{(\KM)}(\M\mid r)$ is the expected loss of the $k$-means clustering using dimensions with $r_j=1$:
\[
L^{(\KM)}(\M\mid r)=E\left[\min_{1 \le l \le k}\sum_{j=1}^pr_{j}(X_{1j}-\mu_{lj})^2\right].
\]
\end{prop}
\begin{proof}
From the basic property of the conditional expectation, 
we can immediately obtain
\begin{align*}
    L^{(\KPOD)}(\M) 
&=
E\left\{E\left[\min_{1 \le l \le k}\sum_{j=1}^pR_{1j}(X_{1j}-\mu_{lj})^2
\;\bigg|\; R_{1}
\right]\right\} \\
&=
\sum_{r \in \{0,1\}^p}\pr(R_{1} = r) L^{(\KM)}(\M\mid r),
\end{align*}
which completes the proof. 
\end{proof}

Now, we will show the inconsistency of the $k$-POD clustering
from the viewpoint of $k$-means clustering.
Write $\widehat{\mathrm{M}}_{\KPOD}\in \arg\min_{\mathrm{M}} \widehat{L}_n^{(\KPOD)}(\mathrm{M})$ 
for an estimator of a cluster center matrix by the $k$-POD clustering.
Write $\mathcal{M}_{\KPOD}^\ast$ for the set of optimal cluster center matrices of the expected loss $L^{(\KPOD)}$ (i.e., 
$\mathcal{M}_{\KPOD}^\ast =\arg\min_{\mathrm{M}} L^{(\KPOD)}(\mathrm{M})$).
We note that $L^{(\KPOD)}$ might have multiple optimal solutions even when $L^{(\KM)}$ has the unique optimal solution up to relabelling.
For example, according to the following theorem, we can assume that the triangular points are close enough to a pair of optimal centers for $L^{(\KPOD)}$ in Figure~\ref{fig:intro}.
In this case, as this Gaussian mixture is symmetric about the $y$-axis, 
the pair of points symmetric to the triangular points across the $y$-axis is also near optimal.

The following theorem shows the convergence of the $k$-POD clustering in the large sample limit.
\begin{theorem}\label{theorem:main}
    Assume that $\|X_1\|$ is bounded almost surely.
    Then we have, as $n$ goes to infinity,
    \[
    L^{(\KPOD)}\big(\widehat{\mathrm{M}}_{\KPOD}\big) \rightarrow \min_{\mathrm{M}} L^{(\KPOD)}(\mathrm{M}) 
    \;\text{ and }\;
    d\big(\widehat{\mathrm{M}}_{\KPOD},\mathcal{M}_{\KPOD}^\ast\big)\rightarrow 0
    \quad a.s.,
    \]
where $d(\M,\mathcal{M}_{\KPOD}^\ast)=\min_{\M^\ast\in \mathcal{M}_{\KPOD}^\ast}\|\M-\M^\ast\|$.
\end{theorem}

In general, since $L^{(\KPOD)}$ and $L^{(\KM)}$ have different optimal solutions, 
$\widehat{\mathrm{M}}_{\KPOD}$ does not converge to an optimal solution of the expected $k$-means loss $L^{(\KM)}$. 
Therefore, the estimated partition of the data space by the $k$-POD clustering is generally different from that by the $k$-means clustering.
On the other hand, for high-dimensional data with few missing components,
the $k$-POD clustering could be a suitable choice.
For high-dimensional data, even if the missing rate of each variable is low, 
the number of complete cases could be very small.
Thus, in such cases,
the $k$-POD clustering provides much better results than the complete-case analysis.

\section{Simulations}

In this section, we illustrate some numerical simulations to verify the inconsistency of $k$-POD. 
We consider the settings of complete data on which $k$-means itself can perform well. 
The Gaussian mixture model $X\sim \sum_{l=1}^{k}\pi_l N(\mu_l^{\ast},\Sigma_l^{\ast})$ is thus used to generate the original complete data, 
where $N(\mu_l^{\ast},\Sigma_l^{\ast})$ is the $p$-dimensional Gaussian distribution with mean $\mu_l^{\ast}$ and covariance $\Sigma_l^{\ast}$, and $\pi_l$ is the mixture weight of the $l$th component. 
The missing completely at random mechanism is considered in this section. 
We generate the original complete data matrix $\mathrm{X}=(X_{ij})_{n\times p}$ from the above mixture model and the indicator matrix $\mathrm{R}=(R_{ij})_{n\times p}$ from the Bernoulli distribution, that is, $R_{1j},\dots,R_{nj}$ is an independent sample from the Bernoulli distribution with the probability of success $q_j \in (0,1]$. Then the incomplete data matrix is generated by $\mathrm{X}$ and $\mathrm{R}$, that is, $X_{ij}$ is observed if $R_{ij}=1$ and $X_{ij}$ is missing if $R_{ij}=0$.  
To measure the bias of the estimator of a cluster center matrix by the $k$-POD clustering $\widehat{\M}_{\KPOD}$, we take the mean square error between $\widehat{\mathrm{M}}_{\KPOD}$ and $\mathrm{M}_{\KM}^{\ast}$ to be the criterion, which is given by
\[
    \text{MSE}(\widehat{\mathrm{M}}_{\KPOD}, \mathrm{M}_{\KM}^{\ast})=\sum_{l=1}^{k} \min_{l'=1,\dots,k}\| \hat{\mu}_{\KPOD,l} - {\mu}_{\KM ,l'}^{\ast} \|^2,
\]
where the minimization with respect to $l'$ is to eliminate the influence of index permutation. 
Since $\mathrm{M}_{\KM}^{\ast}$ is the minimizer of the expected loss of the $k$-means clustering, it is unknown. We here substitute it by the estimator of the $k$-means clustering with the sample size $n=10^5$. 
Since the loss function of the $k$-POD clustering is highly non-convex as with the original $k$-means clustering, we use multiple random initializations and employ the result with the smallest loss value. 
Here, we should note that  \citet{chi2016k} provides the R package \texttt{kpodclustr} including the implementation of the $k$-POD clustering with a single specific initialization, which often provides poor local minima with higher loss values. 

We first verify the inconsistency of the $k$-POD clustering via the trend of MSE as sample size $n$ goes to infinity in Figure~\ref{fig_mse}. 
We consider two settings, each of which consists of $k=3$ clusters with equal component weights and identical component covariance, i.e., $\pi_l=1/3$ and $\Sigma_l^{\ast}=\mathrm{I}_p$ for $l=1,\dots,k$, 
where $\mathrm{I}_p$ is the identity matrix. 
The components' means and observed probabilities are as follows. 
(a) The components means are $\mu_1^{\ast}=(0,0)^{T}$, $\mu_2^{\ast}=(3,0)^{T}$ and $\mu_3^{\ast}=(1.5,(6.75)^{1/2})^{T}$, respectively. The observed probabilities are $q_j=2/3$ for $j=1,2$. 
(b) The components means are $\mu_1^{\ast}=(0,0,0,0,0)^{T}$, $\mu_2^{\ast}=(3,0,0,0,0)^{T}$ and $\mu_3^{\ast}=(1.5,(6.75)^{1/2},0,0,0)^{T}$, respectively. 
The observed probabilities are $q_j=2/3$ for $j=1,\dots,5$. 
For better comparison, we provide the result of $k$-means clustering with complete cases as well as the result of $k$-means clustering with all original data points. 
Since few complete cases are left in the second setting, 
we ignore the result of $k$-means clustering with complete cases. 
It can be seen that the MSE of $k$-means with complete cases (dotted line) would gradually approach that of $k$-means with all original data (solid line). However, the result of $k$-POD clustering (dashed line) would converge but not to zero as the sample size goes to infinity in both settings. The significant gap between the dashed line and the solid line thus implies the inconsistency of the $k$-POD clustering. 

\begin{figure}
     \centering
          \begin{subfigure}[b]{0.3\textwidth}
         \centering
         \includegraphics[width=\textwidth]{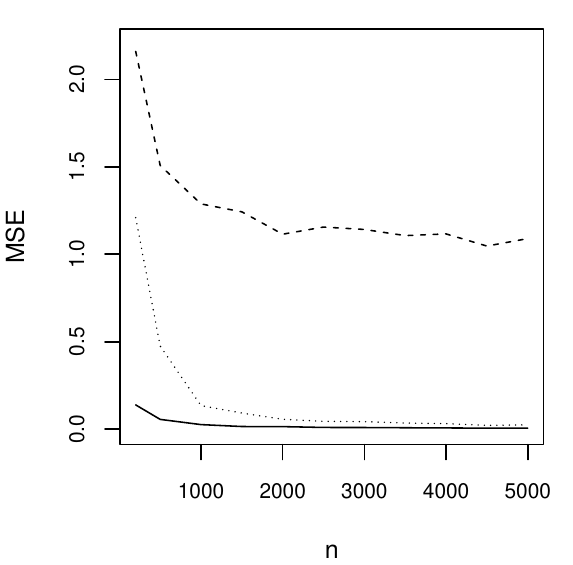}
         \caption{The setting in Section~1}
     \end{subfigure}
     \begin{subfigure}[b]{0.3\textwidth}
         \centering
         \includegraphics[width=\textwidth]{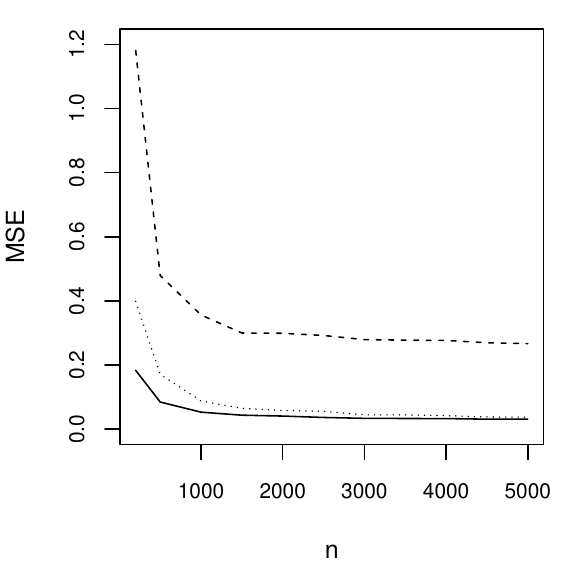}
         \caption{Setting~(a)}
     \end{subfigure}
      \begin{subfigure}[b]{0.3\textwidth}
         \centering
         \includegraphics[width=\textwidth]{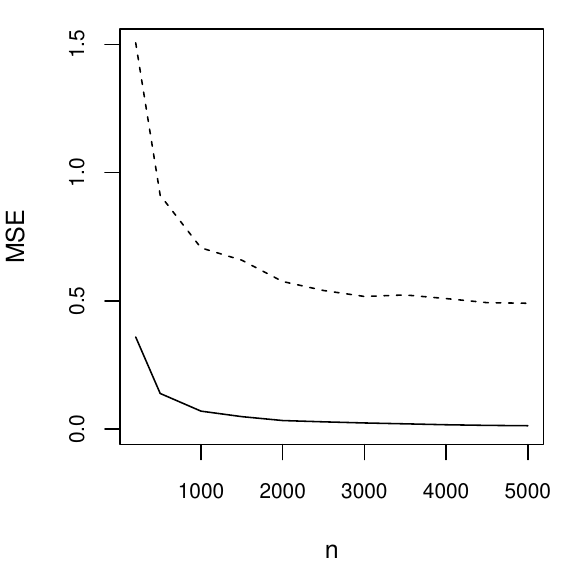}
         \caption{Setting~(b)}
     \end{subfigure}
     \caption{The MSEs of estimated cluster centers 
     (dashed line: the $k$-POD clustering, solid line: the $k$-means clustering, 
     dotted line: the $k$-means clustering with complete cases).}
     \label{fig_mse}
\end{figure}

We further compare the bias of $k$-POD clustering on estimating the cluster centers via several synthetic datasets. 
Table~\ref{table_mse} summarises the results of MSE of different methods on different synthetic datasets, 
that is, \textit{Oracle} ($k$-means clustering with all original data), \textit{Complete-case} ($k$-means clustering with complete cases) and \textit{$k$-POD}. 
The original complete data are also generated from the Gaussian mixture model $\sum_{l=1}^3(1/3)N(\mu^{\ast}_l,\mathrm{I}_p)$.
The synthetic datasets with missing values are all generated by missing completely at random for all dimensions with equal missing probabilities from the generated complete datasets. 
We consider three settings of the Gaussian mixture model for the original complete data. 
\begin{itemize}
\item[(1)] Two-dimensional setting ($p=2$): The three component means are $\mu_1^{\ast}=(0,0)^{T}$, $\mu_2^{\ast}=(3,0)^{T}$, and $\mu_3^{\ast}=(1.5,\sqrt{6.75})^{T}$, respectively. 
\item[(2)] Five-dimensional setting ($p=5$): The first two dimensional elements of the component means are the same as the setting~(1). 
The other components are zero.
\item[(3)] Fifty-dimensional setting ($p=50$): 
The setting is the same as the setting (2) except to the number of dimensions.
\end{itemize}
We set the missing rates of all variables to be the same, 
which varies in $\{10\%, 30\%, 50\%\}$. 
The reported values are the averages and standard deviations of MSEs among 100 repetitions. 
 
\begin{table}
\def~{\hphantom{0}}
\caption{Comparison of MSE of different methods}
\label{table_mse}
\begin{tabular}{crrrcccc}
\hline
Setting & $n$ & $p$ & $k$ &  Missing rate  & Oracle & Complete-case & $k$-POD \\
\hline
(1)		& 3000 & 2 & 3 & 10\% & 0.035 (0.01) & 0.037 (0.01) & 0.070 (0.02) \\
(1)		& 3000 & 2 & 3 & 30\% & 0.035 (0.01) & 0.042 (0.02) & 0.243 (0.05) \\
(1)		& 3000 & 2 & 3 & 50\% & 0.033 (0.01) & 0.067 (0.03) & 0.546 (0.13) \\
(2)		& 5000 & 5 & 3 & 10\% & 0.015 (0.01) & 0.024 (0.01) & 0.043 (0.02) \\
(2)		& 5000 & 5 & 3 & 30\% & 0.014 (0.01) & 0.078 (0.03) & 0.343 (0.08) \\
(2)		& 5000 & 5 & 3 & 50\% & 0.013 (0.01) & 0.421 (0.20) & 1.316 (0.22) \\
(3)		& 10000 & 50 & 3 & 10\% & 0.077 (0.01) & 18.429 (5.49) & 0.117 (0.02)\\
\hline
\end{tabular}
\end{table}

It can be seen that for most datasets, the MSE of the $k$-POD clustering is generally larger than that of others, which indicates the significant bias of estimated cluster centers by the $k$-POD clustering in the application.
On the other hand, for high-dimensional data, even when the missing rate of each variable is low, 
the number of complete cases could be too small, which implies that the complete-case analysis would be less effective. As shown in the last line of Table~\ref{table_mse}, we can see that the $k$-POD clustering performs very well, whereas the complete-case analysis fails due to the small sample size. Therefore, for these cases, the $k$-POD clustering is more stable and could be useful in real applications.

\section{Conclusions}

In this paper, 
we study the theoretical properties of the $k$-means clustering with missing data.
Although the $k$-POD clustering is a natural extension of $k$-means clustering for missing data, 
unfortunately, we show the inconsistency of the $k$-POD clustering 
even under the missing completely at random assumption.
More precisely, as the sample size goes to infinity, the $k$-POD clustering converges to the solution of the weighted sum of the expected losses of the $k$-means clustering with parts of variables.
This result shows that the $k$-POD clustering may fail to capture the hidden cluster structure even when the $k$-means clustering works well on the original complete data.
On the other hand, when the missing rate in each variable is low, 
the $k$-POD clustering is effective for high-dimensional data.



\section*{Acknowledgement}
This research was supported by JSPS, Japan KAKENHI Grant (JP20K19756, JP20H00601, and JP24K14855 to YT), and China Scholarship Council (NO. 202108050077 to XG).
The authors wish to express our thanks to Ms.~Jiayu Li for her helpful discussions.

\appendix 
\section{Proof of Theorem~\ref{theorem:main}}

The proof is similar to the proofs for the consistency of the classical $k$-means.
For simplicity of notation, 
we omit the superscript $(\KPOD)$ from $\widehat{L}_n^{(\KPOD)}(\M)$ and $L^{(\KPOD)}(\M)$ through the appendix. 
Let $C>0$ be the positive constant such that $\|X_1\|<C$ almost surely.
Let $\mathcal{B}(C) := \{x \in \Rb^p\mid \|x\|\le C\}$ be the closed ball. 
By the $k$-POD algorithm, 
we can ensure that the estimated centers of the $k$-POD clustering are in $\mathcal{B}(C)$.
Define $\mathcal{E}_k(C)=\{ \M \in \mathbb{R}^{k\times p} \mid \mu_l\in \mathcal{B}(C),\; l=1\dots,k \}$ for $C>0$.

We write $\widehat{\mathrm{M}}\in \arg\min_{\mathrm{M}\in \mathcal{E}_k(C)} \widehat{L}_n(\mathrm{M})$ 
for an estimator of a cluster center matrix by the $k$-POD clustering, and write $\mathcal{M}^\ast$ for the set of optimal cluster center matrices of the expected loss $L$ (i.e., 
$\mathcal{M}^\ast =\arg\min_{\mathrm{M}\in \mathcal{E}_k(C)} L(\mathrm{M})$). 
Any element of $\mathcal{M}^\ast$ is denoted by $\mathrm{M}^{\ast}$. 
We further write $m_k=\min_{\M\in \mathcal{E}_k(C)} L(\M)$ for the minimal value of the expected loss $L$ 
with $k$ clusters. 
We note that ${\M}^{\ast}$ is not necessarily unique.

Since $\mathcal{B}(C)$ is compact, $\mathcal{E}_k(C)$ is also compact under the topology induced by the Hausdorff metric. 
The following lemma gives the uniform strong law of large numbers and the continuity of the expected loss $L$ on such compact set $\mathcal{E}_k(C)$. 
\begin{lemma}\label{lemma:A1}
Under the assumption in Theorem~\ref{theorem:main}, the followings hold for $C>0$:
\begin{description}
\item{(a)} The uniform law of large numbers holds:
\[
\lim_{n\rightarrow\infty} \sup_{\M\in \mathcal{E}_k(C)} \left| \what{L}_n(\M) - L(\M) \right| =0 \quad a.s.,\;\text{ and}
\]
\item{(b)} The loss function $L(\M)$ is continuous on $\mathcal{E}_k(C)$.
\end{description}
\end{lemma}

Now, we are ready for proving Theorem~\ref{theorem:main}. 

\begin{proof}[Proof of Theorem~\ref{theorem:main}]
The first result is the immediate consequence from (a) of Lemma~\ref{lemma:A1}.
Thus, we will show the second result.
By the optimality of $\what{\M}_n$, we have 
\[
\lim\sup_n \left[ \what{L}_n(\what{\M}_n) - \inf_{\M^\ast \in \mathcal{M}^\ast}\what{L}_n(\M^\ast) \right]
\le 0 \quad \text{a.s.}
\]
By the strong law of large numbers, 
\[
\forall \M^\ast \in \mathcal{M}^\ast;\;\lim\sup_n \inf_{\M^\ast \in \mathcal{M}^\ast}\what{L}_n(\M^\ast)
\le 
\lim\sup_n \what{L}_n(\M^\ast) = L(\M^\ast)\quad\text{a.s.}
\]
Thus, we obtain 
\begin{align*}
0 
&\ge \lim\sup_n  \what{L}_n(\what{\M}_n)  - \lim\sup_n \inf_{\M^\ast \in \mathcal{M}^\ast}\what{L}_n(\M^\ast)
\ge \lim\sup_n  \what{L}_n(\what{\M}_n) - L(\M^\ast)\quad\text{a.s.}
\end{align*}

By the continuity of the expected loss $L$, we have 
\begin{align*}
    \forall \delta >0;\;\min_{\M\in \mathcal{E}_{k,\delta}(C)} L(\M) > \min_{\M \in \mathcal{E}_k(C) } L(\M),
\end{align*}
where $\mathcal{E}_{k,\delta}(C)= \{ \M\in \mathcal{E}_k(C) \mid d(\M, \mathcal{M}^{\ast})\geq \delta \}$. 
Thus, from (a) of Lemma~\ref{lemma:A1}, 
the above inequality leads that for any $\delta>0$ 
\[
\lim\inf_n \inf_{\M \in \mathcal{E}_{k,\delta}(C)}\what{L}_n(\M)
\ge 
\inf_{\M \in \mathcal{E}_{k,\delta}(C)}L(\M)
>
L(\M^\ast)
\ge \lim\sup_n  \what{L}_n(\what{\M}_n)\quad\text{a.s.}
\]
This gives that there exists $n_0 \in \mathbb{N}$ almost surely such that
\[
\forall n \ge n_0;\;\inf_{\M \in \mathcal{E}_{k,\delta}(C)}\what{L}_n(\M) > \what{L}_n(\what{\M}_n).
\]
If $d(\what{\M}_n,\mathcal{M}^\ast)\ge \delta$ for some $n \ge n_0$,
we have 
\[
\inf_{\M \in \mathcal{E}_{k,\delta}(C)}\what{L}_n(\M) > \what{L}_n(\what{\M}_n),
\]
which is impossible. 
Therefore, we conclude that
$\lim_{n\rightarrow \infty} d(\what{\M}_n,\mathcal{M}^\ast) = 0$ \text{a.s.}


\end{proof}

\section{Proofs of Lemma~\ref{lemma:A1}}

Here, we provide the proof of Lemma~\ref{lemma:A1}.  
\begin{proof}[Proof of Lemma~\ref{lemma:A1}]
    For any $\M=(\mu_1,\dots,\mu_k)^{T}\in \mathbb{R}^{k\times p}$, define the function $g_{\M}(\cdot,\cdot):\mathbb{R}^p \times \{0,1\}^p\rightarrow \mathbb{R}$ to be $g_{\M}(x,r)=\min_{1\leq l \leq k} \|x\circ r - \mu_l\circ r\|^2$. 
    Let $\mathcal{G}=\{ g_{\M}(\cdot,\cdot) \mid \M\in \mathcal{E}_k(C) \}$.
    \begin{align*}
        \sup_{g_{\M}\in \mathcal{G}} \left| \frac{1}{n}\sum_{i=1}^{n} g_{\M}(X_i,R_i) - E\left[g_{\M}(X_1,R_1)\right] \right| \rightarrow 0\quad \text{a.s.}
    \end{align*}
    From Theorem~19.4 in \cite{van2000asymptotic}, 
    it suffices to show that for each $\epsilon>0$ 
    there exists a finite class $\mathcal{G}_{\epsilon}$ such that for each $g_{\M}\in \mathcal{G}$, 
    there are functions $\mathring{g}_{\M},\Bar{g}_{\M}\in \mathcal{G}_{\epsilon}$ with $\mathring{g}_{\M}\leq g_{\M}\leq \Bar{g}_{\M}$ and $E\left[ \Bar{g}_{\M}(X_1,R_1)- \mathring{g}_{\M}(X_1,R_1) \right]<\epsilon$. 

    For $\delta>0$, let $D_{\delta}$ be a finite subset of $\mathcal{B}(C)$ such that
    \[
        \forall \mu \in \mathcal{B}(C);\;\exists \nu \in D_{\delta};\; \|\mu-\nu\|<\delta.
    \] 
    Define $\mathcal{D}_{k,\delta}=\{ \M\in \mathcal{E}_k(C) \mid  \mu_l\in D_{\delta},\ l=1,\dots,k\}$. 
    For each $\delta>0$, we give the finite class $\mathcal{G}_{\epsilon}$ of the form:
    \begin{align*}
        \mathcal{G}_{\delta}
        =
        \bigg\{ 
        \min_{1\leq l\leq k} \big(\lVert x\circ r - \nu_l \circ r \rVert \pm \delta \big)^2\;\Bigm|\;  
        \mathrm{V}\in \mathcal{D}_{k,\delta}  
        \bigg\}. 
    \end{align*}
    For a fixed $\M\in \mathcal{E}_k(C)$, 
    take $\mathrm{V}=(\nu_1,\dots,\nu_k)^{T}\in \mathbb{R}^{k\times p}$ 
    such that $\nu_l\in D_{\delta}$ and $\lVert \mu_l - \nu_l \rVert<\delta $ for any $l=1,\dots,k$. 
    Then for $g_{\M}\in \mathcal{G}$, we give the corresponding upper and lower bounds in $\mathcal{G}_{\delta}$ to be 
    \begin{align*}
        \mathring{g}_{\M}(x,r)=\min_{1\leq l\leq k} \big(\lVert x\circ r - \nu_l \circ r \rVert - \delta \big)^2
        \;\text{ and }\;
        \Bar{g}_{\M}(x,r)=\min_{1\leq l\leq k} \big(\lVert x\circ r - \nu_l \circ r \rVert + \delta \big)^2 .
    \end{align*}
    
    We first show that $ \mathring{g}_{\M} \leq g_{\M} \leq \Bar{g}_{\M}$. 
    Since $\mathring{g}_{\M}$ and $\Bar{g}_{\M}$ are determined by $\mathrm{V}$ that satisfies $\lVert \mu_l - \nu_l \rVert<\delta$, we have for any $l=1,\dots,k$ and $(x,r)\in \mathbb{R}^{p}\times \{0,1\}^p$, 
    \begin{align*}
        \lVert x\circ r - \nu_l \circ r \rVert -\delta \leq \lVert x\circ r - \mu_l \circ r \rVert \leq \lVert x\circ r - \nu_l \circ r \rVert +\delta.
    \end{align*}
    It follows that $ \mathring{g}_{\M} \leq g_{\M} \leq \Bar{g}_{\M}$. 
    A simple computation gives
    \begin{align*}
        &E\left[ \Bar{g}_{\M}(X_1,R_1)- \mathring{g}_{\M}(X_1,R_1) \right]\\
        &=\sum_{r\in \{0,1\}^p} \pr(R_1=r)\cdot \int \left\{\Bar{g}_{\M}(x,r)- \mathring{g}_{\M}(x,r)\right\} \,dP(x) \\
        &\leq  \sum_{r\in \{0,1\}^p} \pr(R_1=r)\cdot \int \sum_{l=1}^{k} \left\{ \big(\lVert x\circ r - \nu_l \circ r \rVert + \delta \big)^2 - \big(\lVert x\circ r - \nu_l \circ r \rVert - \delta \big)^2 \right\} dP(x)\\
        &= 4\delta\sum_{r\in \{0,1\}^p} \pr(R_1=r)\cdot \sum_{l=1}^{k}\int  \lVert x\circ r - \nu_l \circ r \rVert  dP(x)
        \le
        4\delta k \left( \int \|x\|  dP(x) + C \right).
    \end{align*}
    This yields $E\big[ \Bar{g}_{\M}(X_1,R_1)- \mathring{g}_{\M}(X_1,R_1) \big]<\epsilon$.

    Next, we prove the continuity of $L(\M)$ on $\mathcal{E}_k(C)$. 
    If $\M, \mathrm{V}\in \mathcal{E}_k(C)$ are chosen to safisfy $\max_{l'} \min_{l} \| \mu_{l'} -\nu_{l} \|<\delta$, 
   \[
    \forall l \in \{1,\dots,k\};\;\exists l'(l)\in \{1,\dots,k\};\;\|\mu_{l'(l)}-\nu_l\|<\delta.
   \] 
    Moreover, we have
    \begin{align*}
        & L(\M)- L(\mathrm{V}) \\
        & = \sum_{r\in \{0,1\}^p} \pr(R_1=r)\cdot 
        \int \left(
        \min_{1\leq l' \leq k} \| x\circ r - \mu_{l'} \circ r \|^2  - \min_{1\leq l \leq k} \| x\circ r - \nu_l \circ r \|^2 
        \right)\,dP(x) 
        \\
        &\leq \sum_{r\in \{0,1\}^p} \pr(R_1=r)\cdot  \int \max_{1\leq l \leq k} \left( \| x\circ r - \mu_{l'(l)} \circ r \|^2 - \| x\circ r - \nu_{l} \circ r \|^2 \right) \ dP(x) \\
        &\leq \sum_{r\in \{0,1\}^p} \pr(R_1=r)\cdot 
        \int \max_{1\leq l \leq k} \left\{ (\| x\circ r - \nu_{l} \circ r \| + \delta)^2 - \| x\circ r - \nu_{l} \circ r \|^2 \right\} \ dP(x) \\
        &\leq 2\delta\sum_{r\in \{0,1\}^p} \pr(R_1=r)\cdot 
        \int \max_{1\leq l \leq k} \| x\circ r - \nu_{l} \circ r \| \ dP(x) + \delta^2
        \leq 4C\delta + \delta^2
    \end{align*}
    Similarly, we obtain $L(\mathrm{V})- L(\M)< 4C\delta + \delta^2$, which completes the proof.  
\end{proof}
\vspace{10pt}

\bibliographystyle{agsm}
\bibliography{main_biblio}
\end{document}